\documentclass[12pt,reqno]{amsart}
\usepackage{amsmath}
\usepackage{amssymb}
\usepackage{amsthm}
\usepackage[english]{babel}
\usepackage{graphics, xcolor}

\newtheorem{theorem}{Theorem}[section]

\newtheorem{lemma}{Lemma}[section]

\newtheorem{proposition}{Proposition}[section]

\newtheorem{corollary}{Corollary}[section]
\theoremstyle{definition}
\newtheorem{definition}{Definition}[section]

\newcommand{\R}{{\mathbf R}}

\newcommand{\N}{{\mathbf N}}

\newcommand{\E}{{\mathbf E\,}}

\newcommand{\beqa}{\begin{eqnarray}}
\newcommand{\beqan}{\begin{eqnarray*}}
\newcommand{\eeqa}{\end{eqnarray}}
\newcommand{\eeqan}{\end{eqnarray*}}
\def\beq#1\eeq{\begin{equation}#1\end{equation}}

\def\P{\mathbf P }

 \def\na{\,\, {\raise.4pt\hbox{$\shortmid$}}{\hskip-2.0pt\to}\, \, }

\def\={\overset{ \text{\rm def} }=}

\def\R{\mathbb R}

\def\4{\kern1pt}

\newcommand{\tc}{}
\newtheorem{remark}{Remark}

\begin{document}

\title[Kolmogorov's uniform limit theorem]
{\tc{An improved multivariate version of Kolmogorov's second uniform limit theorem}}

\author[F.~G\"otze]{Friedrich G\"otze}
\author[A.Yu.~Zaitsev]{Andrei Yu. Zaitsev}
\author[D.~Zaporozhets]{Dmitry Zaporozhets}

\email{goetze@math.uni-bielefeld.de}
\address{Fakult\"at f\"ur Mathematik,\newline\indent
Universit\"at Bielefeld, Postfach 100131,\newline\indent D-33501 Bielefeld,
Germany\bigskip}
\email{zaitsev@pdmi.ras.ru}
\address{St.~Petersburg Department of Steklov Mathematical Institute
\newline\indent
Fontanka 27, St.~Petersburg 191023, Russia\newline\indent
and St.Petersburg State University, 7/9 Universitetskaya nab., St. Petersburg,
199034 Russia}

\begin{abstract}
The aim of the present work is to show that the results obtained earlier
on the approximation of distributions of sums of independent summands by
%the
 \tc{infinitely divisible laws} may be transferred to the estimation of the closeness of distributions on convex polyhedra.
\end{abstract}



\thanks{The authors were supported by SFB 1283.
The first and second authors were supported  by the SPbGU-DFG grant 6.65.37.2017
The second author was supported by grant RFBR 19-01-00356. The second and third authors were supported  by the Program of Fundamental Researches of Russian
Academy of Sciences "Fundamental Mathematics and its Applications", grant PRAS-18-01}

\maketitle

\section{Introduction}
Let us  fix some $n\in\N$ and consider  \emph{independent}  random variables
\begin{align*}
    \xi_1,\dots,\xi_n\in\R^1,
\end{align*}
which are not necessary identically distributed. The classical result of Kolmogorov~\cite{23} states that if $n$ is large, then under quite  general conditions the distribution of their sum is close in the \emph{L\'evy metric} to the class of all infinitely divisible distributions. To be more precise, let us first fix some notation.

For a random variable $\xi\in\R^1$ define its \emph{concentration function} as
\begin{align*}
    Q(\xi;\tau):=\sup_{x\in\R^1}\P[x\leq\xi\leq x+\tau].
\end{align*}
The L\'evy metric between two random variables $\xi,\xi'\in\R^1$ is defined as
\begin{align*}
    L(\xi,\xi'):=\inf\{\lambda>0: \P[\xi\leq x]&\leq \P[\xi'\leq x+\lambda]+\lambda,
    \\\P[\xi'\leq x]&\leq \P[\xi\leq x+\lambda]+\lambda\text{ for all }x\in\R^1\}.
\end{align*}

Denote by $\mathfrak{D}$ the class of all infinitely divisible random variables.

Kolmogorov~\cite{23} showed that there exists an \emph{absolute} constant $c$ such that for \emph{arbitrary} independent random variables $\xi_1,\dots,\xi_n$ and for all $\tau>0$,
\begin{align}\label{1635}
   \inf_{\eta\in \mathfrak{D}} L(\xi_1+\dots+\xi_n,\eta)\leq c\cdot\big(p^{1/5}+\tau^{1/2}(|\log\tau|^{1/4}+1)\big),
\end{align}
where
\begin{align}\label{1500}
    p:=\max\{p_1,\dots,p_n\}\quad\text{and}\quad p_i:=1-Q(\xi_i;\tau),\,\, i=1,\dots,n.
\end{align}
Let us note that the original result in~\cite{23} has a slightly different  form formulated for $\tau\in(0,1/2]$, see Remark~\ref{1228} below.

The restriction \eqref{1500} on the distributions of summands is a \tc{non-asymptotic} analogue of the classical limit constancy condition 
%in the scheme of series of
\tc{for the triangular scheme for independent random
variables}. The bound for the rate of approximations may be considered as a qualitative
improvement of the classical \tc{Khinchin} theorem %about
\tc{for}  the set of infinitely divisible
distributions %as the 
\tc{being  limit laws of the distributions of sums
	in a triangular scheme.} %involved in the scheme of series.%

The L\'evy distance metrizes the weak convergence of probability distributions on the real line.
Therefore, Kolmogorov's inequality \eqref{1635} proves Khinchin's theorem since weak convergence as $p\to0$ and $\tau\to0$ of distributions of sums $\xi_1+\dots+\xi_n$ to some distribution implies weak convergence  to the same %distribution
\tc{limit}  of distributions of some infinitely divisible variables.%The limit distribution 
\tc{This limit} is infinitely divisible as a limit of infinitely divisible distributions. However, Kolmogorov's inequality \eqref{1635} provides good infinitely divisible \tc{approximations} for fixed small $p$ and $\tau$ even if the distributions of sums %involved in the scheme of series 
in the triangular scheme
with $p\to0$ and $\tau\to0$ are not sequentially compact.

Kolmogorov's methods of proving~\eqref{1635} along with a combinatorial lemma of Sperner have been later used by Rogozin in~\cite{bR61, bR61b} to obtain the result which is now well known as %the classical 
Kolmogorov--Rogozin inequality: for some absolute constant $c$ and all $\tau>0$,
\begin{align*}
    Q(\xi_1+\dots+\xi_n;\tau)\leq\frac{c}{\sqrt{p_1+\dots+p_n}}.
\end{align*}

Thereafter, %the 
subsequent improvements of the bound in the right-hand side of~\eqref{1635} have been obtained in~\cite{k63, ip73}. Finally, the optimal %one 
\tc{bound} was derived in Zaitsev and Arak~\cite{49}:
\begin{align}\label{1920}
    \inf_{\eta\in \mathfrak{D}} L(\xi_1+\dots+\xi_n,\eta)\leq c\cdot\big(p+\tau(|\log\tau|+1)\big).
\end{align}
The authors also considered in~\cite{49} the L\'evy--Prokhorov metric
\begin{align*}
    \pi(\xi,\xi'):=\inf\{\lambda\geq0: \P[\xi\in B]&\leq \P[\xi'\in B^\lambda]+\lambda,
    \\\P[\xi'\in B]&\leq \P[\xi\in B^\lambda]+\lambda\text{ for all Borel sets } B\},
\end{align*}
where
\begin{align*}
    B^\lambda:=\{x\in\R^1:\inf_{y\in B}|x-y|<\lambda\},
\end{align*}
and obtained the bound
\begin{align}\label{2234}
    \inf_{\eta\in \mathfrak{D}} \pi(\xi_1+\dots+\xi_n,\eta)\leq c\cdot\big(p+\tau(|\log\tau|+1)\big)+\sum_{i=1}^np_1^2,
\end{align}
which is optimal as well. For a %deeper
\tc{more detailed}  discussion of the subject we refer the reader to~\cite{2}.

Estimates~\eqref{1920} and~\eqref{2234}, although being optimal,
have their drawbacks %,  however: 
\tc{inasmuch they are for instance} not invariant with
respect to the scaling of the random variables. In particular,
they become  %insignificant
\tc{trival} for large $\tau$. Thus in~\cite{z84} 
%it as been suggested
\tc{ the following generalization of~\eqref{1920}
and~\eqref{2234} hasb been suggested which lacks %such kind of
these disadvantages:} for some
absolute constants $c,\varepsilon>0$ and all $\lambda,\tau>0$,
\begin{equation}\label{1059}
     \begin{split}\inf_{\eta\in \mathfrak{D}} L_\lambda(\xi_1+\dots+\xi_n,\eta)&\leq c\cdot\big(p+\exp(-\varepsilon\cdot\lambda/\tau)\big),
     \\\inf_{\eta\in \mathfrak{D}} \pi_\lambda(\xi_1+\dots+\xi_n,\eta)&\leq c\cdot\big(p+\exp(-\varepsilon\cdot\lambda/\tau)\big)+\sum_{i=1}^np_1^2,
\end{split}\end{equation}
where $L_\lambda(\,\cdot\,,\,\cdot\,)$ and
$\pi_\lambda(\,\cdot\,,\,\cdot\,)$ are the following refinements
of the L\'evy and L\'evy--Prokhorov metrics:
\begin{align*}
     L_\lambda(\xi,\xi')&:=\sup_{x\in\R^1}\max\{\P[\xi\leq x]-\P[\xi'\leq x+\lambda], \P[\xi'\leq x]-\P[\xi\leq x+\lambda]\},
     \\ \pi_\lambda(\xi,\xi')&:=\sup_{\substack{\text{Borel}\\\text{sets } B}}\max\{\P[\xi\in B]-\P[\xi'\in B^\lambda], \P[\xi'\in B]-\P[\xi\in B^\lambda]\}.
\end{align*}
The latter quantity was first considered in~\cite{BS80}, while the
former one -- in~\cite{z84}. Knowing
$L_\lambda(\,\cdot\,,\,\cdot\,)$ and
$\pi_\lambda(\,\cdot\,,\,\cdot\,)$ provides more information on
the closeness of the distributions of random variables rather than
just knowing $L(\,\cdot\,,\,\cdot\,)$ and
$\pi(\,\cdot\,,\,\cdot\,)$. In particular, we have
\begin{equation}\label{1105}\begin{split}
     L(\xi,\xi')&:=\inf\{\lambda>0:L_\lambda(\xi,\xi')<\lambda\},
     \\\pi(\xi,\xi')&:=\inf\{\lambda>0:\pi_\lambda(\xi,\xi')<\lambda\}.
\end{split}\end{equation}
\begin{remark}\label{524}
We would like to stress that~\eqref{1059} is indeed the generalization of~\eqref{1920} and~\eqref{2234}: it is straightforward  to check that~\eqref{1059} together with~\eqref{1105} implies~\eqref{1920} and~\eqref{2234}, see \cite{z84} for details.
\end{remark}
\begin{remark}\label{1228}
In~\cite{23},  it was essentially proved  that for all $\tau, \lambda$ such that $0<2\tau\leq \lambda$,
\begin{align*}
     \inf_{\eta\in \mathfrak{D}} L_\lambda(\xi_1+\dots+\xi_n,\eta)\leq c\cdot\bigg( p^{1/5}+\frac\tau\lambda \log^{1/2}\frac \lambda\tau\bigg),
\end{align*}
which together with~\eqref{1105} implies~\eqref{1635} for $\tau\in(0,1/2]$.
\end{remark}

\section{\tc{Higher Dimensions}}
The problem discussed in the previous section can be naturally extended to higher dimensions. Now let
\begin{align*}
    \xi_1,\dots,\xi_n\in\R^d
\end{align*}
be independent $d$-dimensional random vectors. The concentration function of a random vector $\xi\in\R^d$ is defined as
\begin{align}\label{2256}
    Q(\xi;\tau):=\sup_{x\in\R^d}\P[|\xi-x|\leq\tau],
\end{align}
where for $d=1$ this definition differs from the previous one by the scaling factor 2 in the argument. The definitions of $\pi(\,\cdot\,,\,\cdot\,)$ and $\pi_\lambda(\,\cdot\,,\,\cdot\,)$ in $\R^d$ stay without any changes with $B^\lambda$ being defined as
\begin{align*}
    B^\lambda:=\{x\in\R^d:\inf_{y\in B}|x-y|<\lambda\}.
\end{align*}
The situation with $L(\,\cdot\,,\,\cdot\,)$ and $L_\lambda(\,\cdot\,,\,\cdot\,)$ is more ambiguous. In~\cite{z89}, the following multidimensional versions of these quantities have been suggested. Let
\begin{align*}
    \mathbf{1}_d:=(1,\dots,1)\in\R^d.
\end{align*}
For random vectors $\xi,\xi'\in\R^d$ define
\begin{align}\label{1957}
    L(\xi,\xi'):=\inf\{\lambda>0: \P[\xi\leq x]&\leq \P[\xi'\leq x+\lambda\mathbf{1}_d]+\lambda,
    \\\P[\xi'\leq x]&\leq \P[\xi\leq x+\lambda\mathbf{1}_d]+\lambda\text{ for all }x\in\R^d\}\notag
\end{align}
and
\begin{align}\label{2001}
     L_\lambda(\xi,\xi')&:=\sup_{x\in\R^d}\max\{\P[\xi\leq x]-\P[\xi'\leq x+\lambda\mathbf{1}_d], \P[\xi'\leq x]-\P[\xi\leq x+\lambda\mathbf{1}_d]\}.
\end{align}
Consider some arbitrary independent random vectors $\xi_1,\dots,\xi_n\in\R^d$ and let $p, p_1,\dots,p_n$ be defined as in~\eqref{1500}. It was shown in~\cite{z89} that the following multidimensional version of~\eqref{1059} holds: for some constants $c_d,\varepsilon_d>0$ depending on $d$ only and all $\tau,\lambda>0$ we have
\begin{align}\label{1514}
     \inf_{\eta\in \mathfrak{D}_d} L_\lambda(\xi_1+\dots+\xi_n,\eta)&\leq c_d\cdot\big(p+\exp(-\varepsilon_d\cdot\lambda/\tau)\big),
     \\\inf_{\eta\in \mathfrak{D}_d} \pi_\lambda(\xi_1+\dots+\xi_n,\eta)&\leq c_d\cdot\big(p+\exp(-\varepsilon_d\cdot\lambda/\tau)\big)+\sum_{i=1}^np_i^2,\label{2341}
\end{align}
where $\mathfrak{D}_d$ denotes the class of all $d$-dimensional infinitely divisible random vectors.

\begin{remark}\label{529}
Note that~\eqref{1105} remains true in higher dimensions, too. Thus it is not hard to show (see \cite{z84} for details) that~\eqref{1514} together with~\eqref{1105} implies  the  multidimensional analogues of~\eqref{1920} and~\eqref{2234}:
\begin{align}\label{2044}
    \inf_{\eta\in \mathfrak{D}_d} L(\xi_1+\dots+\xi_n,\eta)&\leq c_d\cdot\big(p+\tau(|\log\tau|+1)\big),
    \\\inf_{\eta\in \mathfrak{D}_d} \pi(\xi_1+\dots+\xi_n,\eta)&\leq c_d\cdot\big(p+\tau(|\log\tau|+1)\big)+\sum_{i=1}^np_i^2.\label{2342}
\end{align}
\end{remark}

\section{Main results}
The definition of the multidimensional L\'evy metric given in the previous section is not %quite 
\tc{really} natural: it heavily depends on the choice of the coordinate basis, while the concentration functions involved in the upper bounds in~\eqref{1514} do not. Our aim is to suggest a coordinate-free definition and to obtain a counterpart of~\eqref{1514} for it. Let us start with some notation.

For $m\in\N$ we denote by $\mathcal P_m$ the class of convex polyhedra $P \subset \R^d$
representable as
\begin{equation}\label{wq}
P=\big\{x\in\R^d: \langle x,t_j\rangle\le b_j, \ j=1,\ldots, m\big\},
\end{equation}
where $t_j\in \R^d$ with $|t_j|=1$ and $b_j\in\R^1$, $j=1,\ldots, m$. For $P\in\mathcal P_m$ defined in \eqref{wq} and $\lambda \ge0$ let
\begin{align*}
P_\lambda=\big\{x\in\R^d:\langle x,t_j\rangle\le b_j+\lambda, \ j=1,\ldots, m\big\}.
\end{align*}
By definition, $ P_\lambda$ is the intersection of closed $\lambda$-neighborhoods of half-spaces $\big\{x\in\R^d:\langle x,t_j\rangle\le b_j\big\}$, where $j=1,\ldots, m$. Let us stress that $P_\lambda$ depends on representation~\eqref{wq} which might  be not unique for a given polyhedron.

Let $e_1,\dots,e_d\in \R^d$ be the vectors of the standard Euclidean basis. If we consider  the subclass $\mathcal P_d^*\subset \mathcal P_d$ of those polyhedra  which are representable as
\begin{equation}
P=\big\{x\in\R^d:\langle x,e_j\rangle\le b_j, \ j=1,\ldots, d\big\}, \label{g9}
\end{equation}
then~\eqref{1957} and~\eqref{2001} are obviously equivalent to
\begin{align*}
    L(\xi,\xi'):=\inf\{\lambda\geq0: \P[\xi\in P]&\leq \P[\xi'\in P_\lambda]+\lambda,
    \\\P[\xi'\in P]&\leq \P[\xi\in P_\lambda]+\lambda\text{ for all } P\in \mathcal P_d^*\}
\end{align*}
and
\begin{align}\label{305}
    L_\lambda(\xi,\xi')&:=\sup_{P\in \mathcal P_d^*}\max\{\P[\xi\in P]-\P[\xi'\in P_\lambda], \P[\xi'\in P]-\P[\xi\in P_\lambda]\}.
\end{align}
This observation %leads to the idea
\tc{suggests the following} coordinate-free definition for $L(\,\cdot\,,\,\cdot\,)$ and $L_\lambda(\,\cdot\,,\,\cdot\,)$ in $\R^d$.
\begin{definition}
For $m\in\mathbf N$ and for random vectors $\xi,\xi'\in\R^d$ define
\begin{align*}
    L_m(\xi,\xi'):=\inf\{\lambda\geq0: \P[\xi\in P]&\leq \P[\xi'\in P_\lambda]+\lambda,
    \\\P[\xi'\in P]&\leq \P[\xi\in P_\lambda]+\lambda\text{ for all } P\in \mathcal P_m\}
\end{align*}
and
\begin{align*}
    L_{\lambda, m}(\xi,\xi')&:=\sup_{P\in \mathcal P_m}\max\{\P[\xi\in P]-\P[\xi'\in P_\lambda], \P[\xi'\in P]-\P[\xi\in P_\lambda]\}.
\end{align*}

\end{definition}
It follows directly from the definition that for $m\geq d$,
\begin{align}\label{2145}
    L(\,\cdot\,,\,\cdot\,)\leq L_m(\,\cdot\,,\,\cdot\,)\quad\text{and}\quad L_\lambda(\,\cdot\,,\,\cdot\,)\leq L_{\lambda,m}(\,\cdot\,,\,\cdot\,).
\end{align}
Our first result %gives the
\tc{provides}
 upper bounds for $L_m(\,\cdot\,,\,\cdot\,)$ and $L_{\lambda,m}(\,\cdot\,,\,\cdot\,)$ similar to~\eqref{1514} and~\eqref{2044}.

Consider again some arbitrary independent random vectors $\xi_1,\dots,\xi_n\in\R^d$ and let $p, p_1,\dots,p_n$ be defined as in~\eqref{1500}.
\begin{theorem}\label{th2}
For any $m\in\N$ there exist constants $c_m,\varepsilon_m$ depending on $m$ only such that
\begin{align*}
    \inf_{\eta\in \mathfrak{D}_d} L_m(\xi_1+\dots+\xi_n,\eta)&\leq c_m\cdot\big(p+\tau(|\log\tau|+1)\big),
    \\\inf_{\eta\in \mathfrak{D}_d} L_{\lambda,m}(\xi_1+\dots+\xi_n,\eta)&\leq c_m\cdot\big(p+\exp(-\varepsilon_m\cdot\lambda/\tau)\big),
\end{align*}for all $\tau, \lambda>0$.
\end{theorem}
Let us %stress
\tc{emphasize}  that the constants $c_m, \varepsilon_m$ do not depend on dimension $d$.

Applying Theorem~\ref{th2} for $\tau=0$, and that
\begin{align*}
    \P[\xi\in P]=\lim_{\lambda\to0}\P[\xi\in P_\lambda],\quad\text{ for any } P\in \mathcal P_m,
\end{align*} we get the following Corollary~\ref{c1}.
\begin{corollary}[G\"otze and Zaitsev \cite{GZ17}, see also Zaitsev \cite{z92}]\label{c1}
%Let
\tc{Assume that the conditions of Theorem~$\ref{th2}$ are satisfied} with $\tau=0$. Then, for any $m\in\N$, there exist \tc{a} constant $c_m$ depending on $m$ only such that
$$\inf_{\eta\in \mathfrak{D}_d} \rho_m(\xi_1+\dots+\xi_n,\eta)\leq c_m p,$$ where
\begin{align*}
    \rho_{ m}(\xi,\xi')&:=\sup_{P\in \mathcal P_m}\left|\P[\xi\in P]-\P[\xi'\in P]\right|.
\end{align*}
\end{corollary}In order to prove Corollary~\ref{c1} one should apply Theorem~\ref{th2} as $\tau\to0$ with $\lambda=\sqrt{\tau}\to0$.

In the definition of $L_m(\,\cdot\,,\,\cdot\,)$ and $L_{\lambda,m}(\,\cdot\,,\,\cdot\,)$ we considered convex polyhedra $P$ along with their approximations $P_\lambda$, while in the classical L\'evy--Prokhorov metric \emph{all} Borel sets $B$ along with their neighborhoods $B^\lambda$ are considered. Thus it is natural to examine the intermediate case: convex polyhedra $P$ along with their neighborhoods $P^\lambda$.
\begin{definition}
For $m\in\mathbf N$ and for random vectors $\xi,\xi'\in\R^d$ define
\begin{align*}
    \pi_m(\xi,\xi'):=\inf\{\lambda\geq0: \P[\xi\in P]&\leq \P[\xi'\in P^\lambda]+\lambda,
    \\\P[\xi'\in P]&\leq \P[\xi\in P^\lambda]+\lambda\text{ for all } P\in \mathcal P_m\}
\end{align*}
and
\begin{align*}
    \pi_{\lambda, m}(\xi,\xi')&:=\sup_{P\in \mathcal P_m}\max\{\P[\xi\in P]-\P[\xi'\in P^\lambda], \P[\xi'\in P]-\P[\xi\in P^\lambda]\}.
\end{align*}
\end{definition}
Again, from the definition we have that
\begin{align}\label{2331}
    L(\,\cdot\,,\,\cdot\,)\leq L_m(\,\cdot\,,\,\cdot\,)\leq\pi_m(\,\cdot\,,\,\cdot\,)\leq \pi(\,\cdot\,,\,\cdot\,)\quad\text{and}\quad L_\lambda(\,\cdot\,,\,\cdot\,)\leq L_{\lambda,m}(\,\cdot\,,\,\cdot\,)\leq\pi_{\lambda,m}(\,\cdot\,,\,\cdot\,)\leq\pi_\lambda(\,\cdot\,,\,\cdot\,).
\end{align}
Therefore,~\eqref{2341} and~\eqref{2342} readily give upper bounds for $\pi_m(\,\cdot\,,\,\cdot\,)$  and  $\pi_{\lambda,m}(\,\cdot\,,\,\cdot\,)$. However, as our next theorem shows, the restriction from the class of the Borel sets to the convex polyhedra %lets us dispose of 
\tc{allows us to remove} the term $\sum_{i=1}^np_i^2$ from these bounds.
 
\begin{theorem}\label{th8}
For any $m\in\N$ there exist constants $c_m,\varepsilon_m$ depending on $m$ only such that
\begin{align*}
    \inf_{\eta\in \mathfrak{D}_d} L_m(\xi_1+\dots+\xi_n,\eta)&\leq c_m\cdot\big(p+\tau(|\log\tau|+1)\big),
    \\\inf_{\eta\in \mathfrak{D}_d} L_{\lambda,m}(\xi_1+\dots+\xi_n,\eta)&\leq c_m\cdot\big(p+\exp(-\varepsilon_m\cdot\lambda/\tau)\big),
\end{align*}for all $\tau, \lambda>0$.
\end{theorem}
Again, the constants $c_m, \varepsilon_m$ are dimension-free.

It readily follows from~\eqref{2331} that Theorem~\ref{th8} implies Theorem~\ref{th2}. However, it will be convenient for us first to prove Theorem~\ref{th2} and then to derive Theorem~\ref{th8} from it.

We will derive Theorem~\ref{th2} by means of constructing for any convex polyhedron from $\mathcal P_m$  some linear transformation from $\R^d$ to $\R^m$ so that the polyhedron turns out to be a \tc{pre-image} of another convex polyhedron in $\R^m$ of the form~\eqref{g9},  and then applying~\eqref{1514} and~\eqref{2044}.

The proof of Theorem~\ref{th8} is more involved. To derive it from Theorem~\ref{th2}, we will need to bound $P_\lambda$ by $P^{c\lambda}$ for some $c>1$. In general, %it 
\tc{this is not possible: as it is easily seen}, for any $c>1$ there exists $P\in \mathcal P_m$ such that $P_\lambda\not\subset P^{c\lambda}$. However, \tc{as noted above}, $P_\lambda$ depends on representation~\eqref{wq} which is not unique.  Thus, as the next proposition shows, it will be possible to add to the right-hand side of~\eqref{wq} ``not too many'' half-spaces which do not affect $P$, but change $P_\lambda$  so that $P_\lambda\subset P^{c\lambda}$.

\begin{proposition}\label{l7}
Fix some $m\in\N$ and $\varepsilon>0$. Then there exists a
constant $c_{m,\varepsilon}$ depending on  $m,\varepsilon$ only
such that for any polyhedron  $P\in\mathcal P_m$ of the
form~\eqref{wq}  there exist  $m_0\leq c_{m,\varepsilon}$,
$t_j\in \R^d$ with $|t_j|=1$,  and $b_j\in \R$, $j=m+1,\dots,
m_0$, such that
\begin{align*}
P=\big\{x\in\R^d: \langle x,t_j\rangle\le b_j, \ j=1,\ldots, m_0\big\},
\end{align*}
and for any $\lambda>0$,
\begin{align}\label{2114}
P_\lambda:=\big\{x\in\R^d: \langle x,t_j\rangle\le b_j+\lambda, \ j=1,\ldots, m_0\big\}\subset P^{(1+\varepsilon)\lambda}.
\end{align}
\end{proposition}
The proof of Proposition~\ref{l7} is postponed to Section~\ref{13}. Now let us proceed with the proofs of the theorems.

\section{Remark on the compound Poisson distributions}
Let $\xi\in\R^d$ be a random vector and let $\xi^{(1)},\xi^{(2)},\dots$ be its independent copies. Denote by $e(\xi)$ a random vector in $\R^d$ distributed as
\begin{align*}
    e(\xi)\stackrel{d}{=}\sum_{k=0}^\infty (\xi^{(1)}+\dots+\xi^{(k)})\cdot\mathbf{1}\{\zeta=k\},
\end{align*}
where $\zeta$ has the standard Poisson distribution and is independent of $\xi^{(1)},\xi^{(2)},\dots$. We say that $e(\xi)$ has the \emph{compound Poisson distribution} with respect to $\xi$. Clearly, $e(\xi)$ is infinitely divisible. Every time we construct $e(\xi)$ by some random vector $\xi$, we tacitly assume that it is independent of everything else (including $\xi$).

Le Cam~\cite{LC} has found that the compound Poisson distributions are good candidates for the infinitely divisible approximations of the sums of independent random vectors. This observation provided a new way of obtaining the Kolmogorov-type bounds like~\eqref{1920} and~\eqref{2234}. Let us be more specific.

As above, fix some $\tau>0$, and let
\begin{align*}
    \xi_1,\dots,\xi_n\in\R^d
\end{align*}
be independent $ d $-dimensional random vectors
 and let the quantities $ p, p_1,\dots,p_n $ be defined as in~\eqref{1500}. From~\eqref{2256} and from the definition of $ p $ it follows that there exist {$a'_1,\dots,a'_n\in\R^d$} such that
\begin{align*}
    \P[|\xi_i-{a'_i}|\leq\tau]=1-p_i\geq 1-p,\quad i=1,\dots, n.
\end{align*}

Generalizing the previous one-dimensional results, Zaitsev~\cite{z89} proved the following statement.
\begin{lemma}\label{l1}
Let  $\alpha_i\in\mathbf{R}^1$, $X_i\in\mathbf{R}^d$, $i=1,\ldots,n$, be independent random variables and vectors such that for some $ a'_i \in \mathbf {R} ^ d $
\begin{equation}
\P[\alpha_i=1]=1-\P[\alpha_i=0]=p_i,\quad
\P[|X_i-a'_i|\le\tau|]=1. \quad i=1,\dots, n.\label{e187}
\end{equation}
Let
\begin{equation}
a_i=\E X_i,\quad \xi_i=(1-\alpha_i)X_i+\alpha_i Y_i,\quad i=1,\dots, n,\label{e1877}
\end{equation}
where $Y_i\in\mathbf{R}^d$ are some independent random vectors which are independent of $\{X_i, \alpha_i\}$.
Then for the vector
\begin{align}\label{2322}
\eta_0:=\sum_{i=1}^n\big[a_i+e(\xi_i-a_i)\big].
\end{align}
 and for some constants $ c_d, \varepsilon_d> 0 $, depending on $ d $ only it holds
\begin{align}\label{2319}
L_\lambda(\xi_1+\dots+\xi_n,\eta_0)&\leq c_d\cdot\big(p+\exp(-\varepsilon_d\cdot\lambda/\tau)\big),
\\ \pi_\lambda(\xi_1+\dots+\xi_n,\eta_0)&\leq c_d\cdot\big(p+\exp(-\varepsilon_d\cdot\lambda/\tau)\big)+\sum_{i=1}^np_i^2,\notag
\end{align}
which implies the bounds~\eqref{1514} and ~\eqref{2341}.
\end{lemma}

\section{Proof of Theorems~\ref{th2} and~\ref{th8}}
\begin{proof}[Proof of Theorem~$\ref{th2}$]
Fix some  polyhedron $P\in\mathcal P_m$:
\begin{align*}
    P=\big\{x\in\R^d: \langle x,t_j\rangle\le b_j, \ j=1,\ldots, m\big\}.
\end{align*}
Let $A:\R^d\to\R^m$ be a linear operator mapping as
\begin{align*}
    x\mapsto y=\big(\langle x,t_1\rangle,\dots,\langle x,t_m\rangle\big).
\end{align*}
Let $e_1,\dots, e_m$ be the standard Euclidean basis in $\R^m$. Consider the polyhedron $\widetilde{P}\subset\R^m$ belonging to the class $\mathcal P_m^*$ (see~\eqref{g9}) defined as
\begin{align*}
    \widetilde{P} = \big\{y\in\R^m:\langle y,e_j\rangle\le b_j, \ j=1,\ldots, m\big\}.
\end{align*}
Since
\begin{align*}
    \langle x,t_j\rangle=\langle x,A^* e_j\rangle=\langle Ax,e_j\rangle,
\end{align*}with adjoint operator $A^*:\R^m\to\R^d$, 
it follows that, for any random vector $\xi\in\R^d$, we have
\begin{align*}
    \P[\xi\in P]=  \P[A\xi\in \widetilde{P}]\quad\text{and}\quad\P[\xi\in P_{\lambda}]=  \P[A\xi\in \widetilde{P}_{\lambda}].
\end{align*}
Hence, for any random vectors $\xi,\xi'\in\R^d$ we have
\begin{multline}\label{14}
    \max\{\P[\xi\in P]-\P[\xi'\in P_\lambda], \P[\xi'\in P]-\P[\xi\in P_\lambda]\}
    \\=\max\{\P[A\xi\in \widetilde{P}]-\P[A\xi'\in \widetilde{P}_\lambda], \P[A\xi'\in \widetilde{P}]-\P[A\xi\in \widetilde{P}_\lambda]\}\leq L_{\lambda}(A\xi,A\xi'),
\end{multline}
where in the last step we used~\eqref{305}.

Recall that we consider independent random vectors
$\xi_1,\dots,\xi_n\in\R^d$ and  $p, p_1,\dots,p_n$ \tc{which} are defined as
in~\eqref{1500}. Without loss of generality, we can assume that $
\xi_i $ are represented as in~\eqref{e1877}, where the
distributions of $ X_i $'s coincide with the conditional
distribution of $ \xi_i $'s  provided that $|\xi_i-a'_i|\le\tau$
with some $a'_i\in\mathbf{R}^d$, which exist by~\eqref{1500}.
Moreover, $\P[|\xi_i-a'_i|>  \tau]=p_i$, and the distributions of
$Y_i$'s coincide with the conditional distribution of $ \xi_i$'s
provided that $|\xi_i-a'_i|>\tau$. Let $a_1,\dots, a_n$ and
$\eta_0$ be defined in~\eqref{e1877} and~\eqref{2322}. Since
$|t_j|=1$ for $j=1,\dots,n$, we have $\|A\|\leq \sqrt m$. Using
this fact gives
\begin{align*}
    \P[|AX_i-Aa'_i|\leq \sqrt m\tau]= 1,\quad i=1,\dots, n.
\end{align*}
Moreover, 
\begin{equation}
\E[AX_i]=Aa_i,\quad i=1,\dots, n.\label{e717}
\end{equation}
Notice that
\begin{align*}
    A\eta_0 = \sum_{i=1}^n\big[Aa_i+Ae(\xi_i-a_i)\big]=\sum_{i=1}^n\big[Aa_i+e(A\xi_i-Aa_i)\big].
\end{align*}

Thus, the vectors $ A \xi_1, \dots, A \xi_n $ satisfy all the conditions imposed on the vectors $\xi_1,\dots, \xi_n$ in Lemma \ref{l1} with $ a_i $ replaced by $Aa_i$, $a'_i$ by $Aa'_i$, and $\tau$ by $\tau\sqrt m$.
Therefore, applying~\eqref {2319} to the vectors $A\xi_1,\dots, A\xi_n$ gives (for some constants $ c_m, \varepsilon_m> 0 $, depending on $ m$ only):
\begin{equation}\nonumber
      L_\lambda(A\xi_1+\dots+A\xi_n,A\eta_0)\leq c_m\cdot\big(p+\exp(-\varepsilon_m\cdot\lambda/\tau)\big),
\end{equation}
which together with~\eqref{14} implies
\begin{align*}
      L_{\lambda,m}(\xi_1+\dots+\xi_n,\eta_0)&\leq c_m\cdot\big(p+\exp(-\varepsilon_m\cdot\lambda/\tau)\big).
\end{align*}
Recalling that $\eta_0$  is infinitely divisible finishes the proof of the second inequality of Theorem~\ref{th2}. The first inequality follows from the second one by standard reasoning, see Remarks~\ref{524},~\ref{529}.
\end{proof}

\begin{proof}[Proof of Theorem~$\ref{th8}$]
Fix some  polyhedron $P\in\mathcal P_m$:
\begin{align*}
    P=\big\{x\in\R^d: \langle x,t_j\rangle\le b_j, \ j=1,\ldots, m\big\}.
\end{align*}
It follows from Proposition~\ref{l7} that it is possible to represent $P$ in the form
\begin{align*}
    P=\big\{x\in\R^d: \langle x,t_j\rangle\le b_j, \ j=1,\ldots, m_0\big\}.
\end{align*}
such that
\begin{align*}
    P_{\lambda/2}\subset P^{\lambda}\quad\text{and}\quad m_0\leq N_m\in\N,
\end{align*}
where 
\begin{equation}
  P_{\lambda/2}=\big\{x\in\R^d: \langle x,t_j\rangle\le b_j+\lambda/2, \ j=1,\ldots, m_0\big\}
\end{equation}
and
the constant $N_m$ depends on $m$ only. Thus for any random vectors $\xi,\xi'$ we have
\begin{align*}
    \max&\{\P[\xi\in P]-\P[\xi'\in P^\lambda], \P[\xi'\in P]-\P[\xi\in P^\lambda]\}
    \\&\leq\max\{\P[\xi\in P]-\P[\xi'\in P_{\lambda/2}], \P[\xi'\in P]-\P[\xi\in P_{\lambda/2}]\}\leq L_{\lambda/2,N_m}(\xi,\xi').
\end{align*}
Since this holds for any $P\in\mathcal P_m$\ we arrive at 
\begin{align*}
    \pi_{\lambda, m}(\,\cdot\,,\,\cdot\,)\leq L_{\lambda/2,N_m}(\,\cdot\,,\,\cdot\,).
\end{align*}
Thus the second inequality of Theorem~\ref{th8} follows from the second inequality of Theorem~\ref{th2}. The constants depending on $N_m$ may be treated as constants depending on $m$. The first part follows from the second one by standard reasoning, see Remarks~\ref{524},~\ref{529}.
\end{proof}

\section{Proof of Proposition~\ref{l7}}\label{13}
\begin{proof}
First let us construct $b_{m+1}, t_{m+1}, \dots, b_{m_0}, t_{m_0}$. Let
\begin{align*}
H_j:=\{x\in\mathbf R^d: \langle x,t_j\rangle= b_j\},\quad j=1,\ldots, m.
\end{align*}
  Denote by $\mathcal F_k(P)$  the collection of all $ k $-dimensional faces of  the \tc{polyhedron} $P$. If $k<d-\min(d,m)$, then $\mathcal F_k(P)$ is empty. Therefore,
\begin{align*}
\mathcal F(P):=\bigcup_{k=d-\min(d,m)}^{d-1}\mathcal F_k(P)
\end{align*}
 is the collection of all proper faces of $P$. Fix some face $F\in\mathcal F(P)$  and let $x_F$ be some point from its relative interior, $\mbox{relint}F$. Denote by $T_F$ the tangent cone at the face $F$ which is defined as
\begin{align*}
T_F=T_F(P):=\{x\in\R^d:x_F+\varepsilon x\in P\text{ for some } \varepsilon>0\}.
\end{align*}
Obviously, $T_F$ does not depend on the choice of $x_F$. Also define by $N_F$ the normal cone at the face $F$ which is defined as the polar cone of $T_F$:
\begin{align*}
N_F=N_F(P):=\{x\in\R^d:\langle x,y\rangle\leq0\text{ for all } y\in T_F\}.
\end{align*}
It is known that
\begin{align}\label{556}
\dim N_F\leq m.
\end{align}
Let $\delta>0$ be some  small enough constant to be fixed later.  Since $N_F$ intersected with the unit sphere $\mathbb S^{d-1}$ is compact, there exists a finite covering subset $\mathcal S_F\subset N_F\cap \mathbb S^{d-1}$ such that
\begin{align}\label{2129}
\max_{v\in\mathcal{S}_F}\langle u, v\rangle\geq 1-\delta\quad\text{for any}\quad u\in N_F\cap \mathbb S^{d-1}.
\end{align}
Moreover, it follows from~\eqref{556} that $\mathcal{S}_F$ is contained in the $(\min(d,m)-1)$-dimensional unit sphere, so
\begin{align}\label{2108}
    \# \mathcal{S}_F\leq c_{m,\delta},
\end{align}
where  $c_{m,\delta}$ depends on $m,\delta$ only.

Now define the desirable set $\{(b_j,t_j)\}_{j=m+1}^{m_0}$ to be a set of all pairs
\begin{align*}
(\langle x_F,v\rangle, v),
\end{align*}
where $v$ %goes 
\tc{runs} over all points of $\mathcal S_F$, while $F$ %goes
\tc{runs}  over all faces of $P$. Again, let us %stress
\tc{emphasize}  that this definition does not depend on the choices of $x_F$'s. Since any face is represented as an intersection of $P$ with several hyperplanes from $H_1,\dots, H_m$, their number is at most $2^m$, which together with~\eqref{2108} implies
\begin{align*}
m_0\leq m+ 2^m c_{m,\delta}.
\end{align*}
Now having defined $b_{m+1}. t_{m+1}, \dots, b_{m_0}, t_{m_0}$, let us prove~\eqref{2114}. Fix some point $x_0\in P_\lambda$. Let $y_0$ denote the Euclidean projection of $x_0$ onto $P$:
\begin{align}\label{2113}
y_0:=\mathrm{argmin}\{\|x_0-y\|:y\in P\}.
\end{align}
The task is to show that
\begin{align}\label{2156}
\|x_0-y_0\|\leq (1+\varepsilon)\lambda.
\end{align}
We obviously may assume that $x_0\not\in P$ which implies $y_0\in\partial P$. It is well-known that the boundary of a polyhedron can be represented as a union of the relative interiors of its proper faces (which do not intersect). Thus there is a unique face $F$ of $P$ such that
\begin{align*}
y_0\in\mathrm{relint}F.
\end{align*}
The latter together with~\eqref{2113} implies
\begin{align*}
x_0-y_0\in N_F.
\end{align*}
Thus it follows from~\eqref{2129} applying to $(x_0-y_0)/\|x_0-y_0\|$ that for some $v\in \mathcal S_F$,
\begin{align*}
\langle x_0-y_0, v\rangle\geq (1-\delta)\|x_0-y_0\|.
\end{align*}
At the other hand, since $x_0\in P_\lambda$ and by the construction of $\{(b_j,t_j)\}_{j=m+1}^{m_0}$,
\begin{align*}
\langle x_0, v\rangle\leq  \langle y_0, v\rangle+\lambda.
\end{align*}
Combining the last two inequalities gives
\begin{align*}
\|x_0-y_0\|\leq\frac{\lambda}{1-\delta}.
\end{align*}
Choosing $\delta=1-\frac1{1+\varepsilon}$ implies~\eqref{2156}, and the proposition follows.
\end{proof}

\end{document}